\documentclass{article}

\usepackage{graphicx}
\usepackage{indentfirst}
\usepackage{amsmath,amsfonts,amsthm,amssymb}
\usepackage{mathrsfs}
\usepackage{amscd}

\def\co{\colon\thinspace}
\def\ond{\hspace{-2pt}\stackrel{\circ}{\nu}\hspace{-3pt}}
\DeclareMathAlphabet{\mathsfsl}{OT1}{cmss}{m}{sl}

\newtheorem{thm}{Theorem}[section]
\newtheorem{lem}[thm]{Lemma}

\newtheorem{prop}[thm]{Proposition}
\newtheorem*{thm*}{Theorem}

\theoremstyle{definition}
\newtheorem{defn}[thm]{Definition}

\begin{document}

\title{Non-separating spheres and twisted Heegaard Floer homology}

\author{{Yi NI}\\{\normalsize Department of Mathematics, MIT, Room 2-306}\\
{\normalsize 77 Massachusetts Avenue, Cambridge, MA
02139-4307}\\{\small\it Emai\/l\/:\quad\rm yni@math.mit.edu}}

\date{}
\maketitle

\begin{abstract}
If a $3$--manifold $Y$ contains a non-separating sphere, then some
twisted Heegaard Floer homology of $Y$ is zero. This simple fact
allows us to prove several results about Dehn surgery on knots in
such manifolds. Similar results have been proved for knots in
$L$--spaces.
\end{abstract}

\section{Introduction}

Heegaard Floer homology was introduced by Ozsv\'ath and Szab\'o
\cite{OSzAnn1}. For null-homologous knots, there is a filtered
version of Heegaard Floer homology, called knot Floer homology
\cite{OSzKnot,RasThesis}. Basically, if one knows the information
about the knot Floer homology of a knot, then one can compute the
Heegaard Floer homology of any manifold obtained by Dehn surgery
on the knot. However, in general the algebra involved here is too
complicated. In order to get useful information, people often
assume the ambient manifold has ``simple" Heegaard Floer homology,
namely, the ambient manifold is an $L$--space.

This paper is motivated by the observation that if the ambient
manifold contains a non-separating sphere, and if we use twisted
coefficients over a Novikov ring, then the Heegaard Floer homology
of the ambient manifold is even simpler: in this case the twisted
Heegaard Floer homology is zero. This observation allows us to
prove several results about null-homologous knots in such ambient
manifolds.

In order to state the first theorem, we introduce the concept of
``Property G".

Suppose $K$ is a null-homologous knot in a closed $3$--manifold
$Y$, then there is a canonical ``zero" slope on $K$. Let $Y_0(K)$
be the manifold obtained from $Y$ by the zero surgery on $K$. (In
general, let $Y_r(K)$ be the manifold obtained from $Y$ by
$r$--surgery on $K$.) Gabai proved the following result in
\cite{G3}:

\begin{thm*}[Gabai]
Let $K$ be a knot in $S^3$, $F$ be a minimal genus Seifert surface
for $K$. Let $\widehat F\subset S^3_0(K)$ be the surface obtained
by capping off $\partial F$ with a disk, then $\widehat F$ is
Thurston norm minimizing in $S^3_0(K)$. Moreover, if $S^3_0(K)$
fibers over the circle, then $K$ is a fibered knot.
\end{thm*}

Our notion of ``Property G" is motivated by the above theorem.

\begin{defn}
Suppose $K$ is a null-homologous knot in a closed $3$--manifold
$Y$. An oriented surface $F\subset Y$ is a {\it Seifert-like
surface} for $K$, if $\partial F=K$. When $F$ is connected, we say
that $F$ is a {\it Seifert surface} for $K$. We also view a
Seifert-like surface as a proper surface in $Y-\ond(K)$.
\end{defn}

\begin{defn}
Suppose $M$ is a compact $3$--manifold, a properly embedded
surface $S\subset M$ is {\it taut} if $x(S)=x([S])$ in
$H_2(M,\partial S)$, $S$ is incompressible,
 and no proper subsurface of
$S$ is null-homologous. Here $x(\cdot)$ is the Thurston norm.
\end{defn}

\begin{defn}
Suppose $K$ is a null-homologous knot in a closed $3$--manifold
$Y$. We say $K$ has {\it Property G}, if the following conditions
hold:

\noindent(G1) any taut Seifert-like surface for $K$ extends to a
taut surface in $Y_0(K)$ after attaching a disk to its boundary;

\noindent(G2) if $Y_0(K)$ fibers over $S^1$, such that the
homology class of the fiber is the extension of the homology class
of a Seifert surface $F$ for $K$, then $K$ is a fibered knot, and
the homology class of the fiber is $[F]$.

If the first (or second) condition holds, then we say that $K$ has
{\it Property G1 ({\rm or} G2)}.
\end{defn}

It is easy to construct knots that violate Property G. However, if
we make some assumption on $Y$ or $K$, then we can get Property G.
For example, one can show that non-prime knots have Property G. In
\cite{G3}, Gabai proved that if $K$ is a null-homologous knot in a
reducible manifold $Y$, such that $H_1(Y)$ is torsion-free and
$Y-K$ is irreducible, then $K$ has Property G. This result has
 overlap with our Theorem~\ref{thm:PropG}. Moreover, using
Heegaard Floer homology, we can show that if
$HF_{\mathrm{red}}(Y)=0$ then $K$ has Property G. (For Property
G2, the proof can be found in \cite{NiFibred,AiNi}. The proof for
Property G1 is similar.)

The first main theorem in this paper is Property G for knots in
manifolds that contain non-separating spheres.

\begin{thm}\label{thm:PropG}
Suppose $Y$ is a closed $3$--manifold that contains a
non-separating sphere $S$, $K$ is a null-homologous knot in $Y$,
such that $Y-K$ is irreducible. Then $K$ has Property G.
\end{thm}

The next result is about cosmetic surgeries on the above knots,
which is an analogue of \cite[Theorem~9.7]{OSzRatSurg}.

\begin{thm}\label{thm:Cosmetic}
Suppose $Y$ is a closed $3$--manifold that contains a
non-separating sphere $S$, $K$ is a null-homologous knot in $Y$,
such that $Y-K$ is irreducible. If two rational numbers $r,s$
satisfy that $Y_{r}(K)\cong\pm Y_{s}(K)$, then $r=\pm s$.
\end{thm}

The paper is organized as follows. In Section~2 we define a
version of twisted Heegaard Floer homology. In Section~3 we
collect some properties of twisted Heegaard Floer homology,
especially the nontriviality results. Sections~4 and 5 are devoted
to the proof of our main theorems.

\vspace{5pt}\noindent{\bf Acknowledgements.}\quad We are very
grateful to David Gabai and Cameron Gordon for helpful
communications. The author is partially supported by an AIM
Five-Year Fellowship and NSF grant number DMS-0805807.

\section{Preliminaries on twisted Heegaard Floer homology}

In this section, we will set up the version of twisted Heegaard
Floer homology we need. Our approach is similar to the sketch in
\cite{NiClosed}. More general constructions can be found in
\cite{OSzAnn2,JM}.

\subsection{Twisted chain complexes}

Let $Y$ be a closed, oriented $3$--manifold.
$(\Sigma,\mbox{\boldmath$\alpha$},\mbox{\boldmath$\beta$},z)$ is a
Heegaard diagram for $Y$. We always assume the diagram satisfies a
certain admissibility condition so that the Heegaard Floer
invariants we are considering are well-defined (see \cite{OSzAnn1}
for more details).

Let
$$\Lambda=\left\{\sum_{r \in \mathbb{R}}a_r T^r\bigg|a_r\in\mathbb{R},
  \;\#\{a_r|a_r\ne0, r \leq c\}<\infty\quad \text{\rm for any $c\in\mathbb R$}
  \right\}$$
be the {\it universal Novikov ring}, which is actually a field.

Let $\omega$ be a $1$--cycle on $\Sigma$, such that it is in
general position with the $\alpha$-- and $\beta$--curves. Namely,
$\omega=\sum k_ic_i$, where $k_i\in\mathbb R$, each $c_i$ is an
immersed closed oriented curve on $\Sigma$, such that $c_i$ is
transverse to the $\alpha$-- and $\beta$--curves, and $c_i$ does
not contain any intersection point of $\alpha$-- and
$\beta$--curves. We also regard $\omega$ as a $1$--cycle in $Y$.

Let $\underline{CF^{\infty}}(Y,\omega;\Lambda)$ be the
$\Lambda$--module freely generated by $[\mathbf x,i]$, where
$\mathbf x\in\mathbb T_{\alpha}\cap\mathbb T_{\beta}$,
$i\in\mathbb Z$. If $\phi$ is a topological Whitney disk
connecting $\mathbf x$ to $\mathbf y$, let
$\partial_{\alpha}\phi=(\partial\phi)\cap\mathbb T_{\alpha}$. We
can also regard $\partial_{\alpha}\phi$ as a multi-arc that lies
on $\Sigma$ and connects $\mathbf x$ to $\mathbf y$. We define
$$A(\phi)=(\partial_{\alpha}\phi)\cdot\omega.$$

 Let
$$\underline{\partial}\co \underline{CF^{\infty}}(Y,\omega;\Lambda)\to
\underline{CF^{\infty}}(Y,\omega;\Lambda)$$ be the boundary map
defined by
$$\underline{\partial}\:[\mathbf x,i]=\sum_{\mathbf y}\sum_{\stackrel{\scriptstyle\phi\in\pi_2(\mathbf x,\mathbf y)}
{\mu(\phi)=1}}\#\big(\mathcal M(\phi)/\mathbb R\big)
T^{A(\phi)}[\mathbf y,i-n_z(\phi)].$$

\begin{prop}\label{prop:TwistInv}
If $\omega_1,\omega_2$ are two $1$--cycles which are homologous in
$Y$, then we have the isomorphism of chain complexes
$$\underline{CF^{\infty}}(Y,\omega_1;\Lambda)\cong\underline{CF^{\infty}}(Y,\omega_2;\Lambda).$$
In particular, when $\omega$ is null-homologous in $Y$, the
coefficients are ``untwisted".
\end{prop}
\begin{proof}
Since $\omega_1,\omega_2$ are homologous in $Y$,
$\omega_1-\omega_2$ is homologous to a linear combination of
$\alpha$--curves and $\beta$--curves in $\Sigma$. It is easy to
check that $\partial_{\alpha}\phi\cdot\gamma=0$ whenever $\phi$ is
a Whitney disk and $\gamma$ is a parallel copy of an $\alpha$-- or
$\beta$--curve. Hence we may assume that $\omega_1-\omega_2$ is
null-homologous in $\Sigma$.

Let $D$ be a $2$--chain in $\Sigma$ such that $\partial
D=\omega_1-\omega_2$. Consider the map
$$\begin{array}{lccc}
f\co&\underline{CF^{\infty}}(Y,\omega_1;\Lambda)&\to&\underline{CF^{\infty}}(Y,\omega_2;\Lambda),\\
&\mathbf x&\mapsto &T^{D\cdot\mathbf x}\mathbf x
\end{array}$$
where $D\cdot\mathbf x$ is the cap product of $D$ with the
$0$--chain $\sum x_i$ if $\mathbf x=(x_1,\dots,x_g)$. We can check
that $f$ is a chain map which induces an isomorphism.
\end{proof}

The standard construction in Heegaard Floer homology
\cite{OSzAnn1} allows us to define the chain complexes
$\underline{\widehat{CF}}(Y,\omega;\Lambda)$ and
$\underline{CF^{\pm}}(Y,\omega;\Lambda)$. The homologies of the
chain complexes are called twisted Heegaard Floer homologies.
Proposition~\ref{prop:TwistInv} allows us to regard $\omega$ as a
homology class in $H_1(Y;\mathbb R)$.

This version of twisted Heegaard Floer homology is a special case
of the general construction in \cite[Section 8]{OSzAnn2}. In fact,
given a $1$--cycle $\omega$, $\Lambda$ can be viewed as a module
over the group ring $\mathbb Z[H^1(Y;\mathbb Z)]$, where the
action of $h\in H^1(Y;\mathbb Z)$ on $T^r\in\Lambda$ is given by
$$h\cdot T^r=T^{r+\langle h,\omega\rangle}.$$ One can check that
the twisted Floer homology defined above is exactly the twisted
Floer homology over the module $\Lambda$ as defined in
\cite[Section 8]{OSzAnn2}.

\begin{prop}\label{prop:OrRev}
Let $Y$ be a $3$--manifold, $\mathfrak s$ be a Spin$^c$ structure,
and $\omega$ be a $1$--cycle. Then, there are natural
isomorphisms:
$$\underline{\widehat{HF}^*}(Y,\omega,\mathfrak s)\cong\underline{\widehat{HF}_*}(-Y,\omega,\mathfrak s),\quad
\underline{HF^*_{\pm}}(Y,\omega,\mathfrak s)\cong
\underline{HF_*^{\pm}}(-Y,\omega,\mathfrak s).$$
\end{prop}
\begin{proof}
As in \cite[Proposition~2.5]{OSzAnn2}, if
$(\Sigma,\mbox{\boldmath$\alpha$},\mbox{\boldmath$\beta$})$ is a
Heegaard diagram for $Y$, then
$(-\Sigma,\mbox{\boldmath$\alpha$},\mbox{\boldmath$\beta$})$ is a
Heegaard diagram for $-Y$. Suppose $\phi\in\pi_2(\mathbf x,\mathbf
y)$ for $Y$, then there is a corresponding $\phi'\in\pi_2(\mathbf
y,\mathbf x)$ for $-Y$. Moreover,
$$\mathcal M_{J_s}(\phi)\cong\mathcal M_{-J_s}(\phi'),\quad\partial_{\alpha}(\phi)=-\partial_{\alpha}(\phi').$$ We then have
$$\big(\partial_{\alpha}(\phi)\cdot\omega\big)_{\Sigma}=\big(\partial_{\alpha}(\phi')\cdot\omega\big)_{-\Sigma}.$$
Now we can easily get our conclusion.
\end{proof}

\subsection{Twisted chain maps}

Let
$(\Sigma,\mbox{\boldmath$\alpha$},\mbox{\boldmath$\beta$},\mbox{\boldmath$\gamma$},z)$
be a Heegaard triple-diagram. Let $\omega$ be a $1$--cycle on
$\Sigma$ which is in general position with the $\alpha$--,
$\beta$-- and $\gamma$--curves.

The pants construction in \cite[Subsection 8.1]{OSzAnn1} gives
rise to a four-manifold $X_{\alpha,\beta,\gamma}$ with
$$\partial
X_{\alpha,\beta,\gamma}=-Y_{\alpha,\beta}-Y_{\beta,\gamma}+Y_{\alpha,\gamma}\:.$$
By this construction $X_{\alpha,\beta,\gamma}$ contains a region
$\Sigma\times \triangle$, where $\triangle$ is a two-simplex with
edges $e_{\alpha},e_{\beta},e_{\gamma}$. Let $\omega\times
[0,1]=\omega\times e_{\alpha}\subset X_{\alpha,\beta,\gamma}$ be
the linear combination of properly immersed annuli such that
$$\omega\times\{0\}\subset
Y_{\alpha,\beta}\:,\quad\omega\times\{1\}\subset
Y_{\alpha,\gamma}\:.$$

Suppose $\mathbf x\in\mathbb T_{\alpha}\cap\mathbb T_{\beta},
\mathbf y\in\mathbb T_{\beta}\cap\mathbb T_{\gamma},\mathbf
w\in\mathbb T_{\alpha}\cap\mathbb T_{\gamma}$, $\psi$ is a
topological Whitney triangle connecting them. Let
$\partial_{\alpha}\psi=\partial\psi\cap\mathbb T_{\alpha}$ be the
arc connecting $\mathbf x$ to $\mathbf w$. We can regard
$\partial_{\alpha}\psi$ as a multi-arc on $\Sigma$. Define
$$A_3(\psi)=(\partial_{\alpha}\psi)\cdot\omega.$$

Let the chain map
$$\underline{f^{\infty}_{\alpha,\beta,\gamma,\:\omega\times I}}\co \underline{CF^{\infty}}(Y_{\alpha,\beta},\omega\times\{0\};\Lambda)
\otimes_{\mathbb Q} CF^{\infty}(Y_{\beta,\gamma};\mathbb R)\to
\underline{CF^{\infty}}(Y_{\alpha,\gamma},\omega\times\{1\};\Lambda)$$
be defined by the formula:
$$\underline{f^{\infty}_{\alpha,\beta,\gamma,\:\omega\times I}}([\mathbf x,i]\otimes[\mathbf y,j])=
\sum_{\mathbf w}\sum_{\stackrel{\scriptstyle\psi\in\pi_2(\mathbf
x,\mathbf y,\mathbf w )}{\mu(\psi)=0}}\#\mathcal
M(\psi)T^{A_3(\psi)}[\mathbf w,i+j-n_z(\psi)].$$

The standard constructions \cite{OSzAnn1,OSzAnn2} allow us to
define chain maps introduced by cobordisms.

\subsection{Twisted Knot Floer homology}

Suppose $K$ is a rationally null-homologous oriented knot in $Y$,
$\xi$ is a relative $\mathrm{Spin}^c$--structure in
$\underline{\mathrm{Spin}^c}(Y,K)$ and $\omega$ is a $1$--cycle in
$Y-K$, we can define the twisted knot Floer complex
$\underline{CFK^{\infty}}(Y,K,\xi,\omega;\Lambda)$ as in
\cite[Section~3]{OSzRatSurg}, see also \cite{OSzKnot,RasThesis}.
Recall that the chain complex is generated by the $[\mathbf
x,i,j]$'s satisfying
\begin{equation}\label{eq:Complex}
\underline{\mathfrak s}_{w,z}(\mathbf
x)+(i-j)\cdot\mathrm{PD}[\mu]=\xi.
\end{equation}

Since $K$ is oriented, there is a natural way to extend a vector
field representing a relative $\mathrm{Spin}^c$--structure in
$\underline{\mathrm{Spin}^c}(Y,K)$ to a vector field on $Y$. Let
$$G_{Y,K}\co\underline{\mathrm{Spin}^c}(Y,K)\to\mathrm{Spin}^c(Y)$$
be the induced map of $\mathrm{Spin}^c$--structures.

\begin{lem}\label{lem:Forget}{\rm\cite[Proposition~3.2]{OSzRatSurg}}There are natural isomorphisms of chain complexes:
$$\underline{C}_{\xi}\{i=0\}\cong\underline{\widehat{CF}}(Y,G_{Y,K}(\xi)),\quad
\underline{C}_{\xi}\{j=0\}\cong\underline{\widehat{CF}}(Y,G_{Y,-K}(\xi)).$$
\end{lem}

We can construct a Heegaard diagram
$(\Sigma,\mbox{\boldmath$\alpha$},\mbox{\boldmath$\beta$},w,z)$
for $(Y,K)$, such that $\beta_1=\mu$ is the meridian of $K$, and
$\alpha_1$ is the only $\alpha$--curve that intersects $\beta_1$,
$\alpha_1\cap\beta_1=\{x\}$. There is a curve $\lambda\subset
\Sigma$ which gives rise to the knot $K$.
$(\Sigma,\mbox{\boldmath$\alpha$},\mbox{\boldmath$\gamma$},z)$ is
a diagram for $Y_{m\mu+\lambda}(K)$, where $\gamma_1=m\mu+\lambda$
and all other $\gamma_i$'s are small Hamiltonian translations of
$\beta_i$'s. Figure~1 (which is a modification of
\cite[Figure~1]{OSzRatSurg}) is the local picture in a cylindrical
neighborhood of $\beta_1$.

\begin{figure}
\begin{picture}(340,100)
\put(0,0){\scalebox{0.64}{\includegraphics*[25pt,380pt][560pt,
530pt]{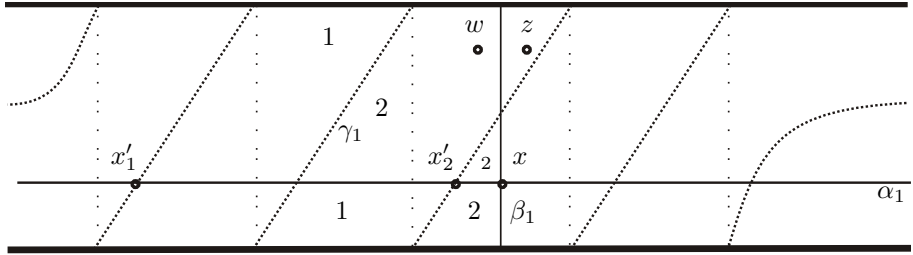}}}

\put(330,22){$\alpha_1$}

\put(191,14){$\beta_1$}

\put(174,85){$w$}

\put(195,85){$z$}

\put(192,35){$x$}

\put(160,35){$x'_2$}

\put(40,35){$x'_1$}

\put(140,53){$2$}

\put(175,14){$2$}

\put(180,33){\scriptsize$2$}

\put(120,80){$1$}

\put(125,14){$1$}

 \put(126,45){$\gamma_1$}

\end{picture}
\caption{Local picture of the triple Heegaard diagram}
\end{figure}

As in \cite{OSzRatSurg}, when $m$ is sufficiently large, one
defines a map
$$\Xi\co\mathrm{Spin}^c(Y_{m\mu+\lambda}(K))\to\underline{\mathrm{Spin}^c}(Y,K)$$
as follows. If $\mathfrak
t\in\mathrm{Spin}^c(Y_{m\mu+\lambda}(K))$ is represented by an
 point $\mathbf x'$ supported in
the winding region, let $\mathbf x\in\mathbb T_{\alpha}\cap\mathbb
T_{\beta}$ be the ``nearest point", and let $\psi\in\pi_2(\mathbf
x',\Theta,\mathbf x)$ be a small triangle. Then
\begin{equation}\label{eq:Xi}
\Xi(\mathfrak t)=\underline{\mathfrak s}_{w,z}(\mathbf
x)+\big(n_w(\psi)-n_z(\psi)\big)\cdot\mu.
\end{equation}

\begin{lem}\label{lem:XiInj}
The map $\Xi$ is injective.
\end{lem}
\begin{proof}
Suppose two intersection points $\mathbf x'_1,\mathbf
x'_2\in\mathbb T_{\alpha}\cap\mathbb T_{\gamma}$ are supported in
the winding region, and they represent two Spin$^c$--structures
$\mathfrak t_1,\mathfrak t_2\in\mathrm{Spin}^c(Y_{m\mu+\lambda})$.
 Let $\mathbf x_1,\mathbf x_2\in\mathbb T_{\alpha}\cap\mathbb T_{\beta}$ be the
nearest points of $\mathbf x'_1,\mathbf x'_2$, and let
$\psi_1,\psi_2$ be the corresponding small triangles.

Assume that $\Xi(\mathfrak t_1)=\Xi(\mathfrak t_2)$. By Equation
(\ref{eq:Xi}), we have
\begin{equation}\label{eq:SpinEq}
\underline{\mathfrak s}_{w,z}(\mathbf
x_1)+\big(n_w(\psi_1)-n_z(\psi_1)\big)\cdot\mu=\underline{\mathfrak
s}_{w,z}(\mathbf x_2)+\big(n_w(\psi_2)-n_z(\psi_2)\big)\cdot\mu.
\end{equation}
Since $\mu$ is null-homologous in $Y$, $\mathbf x_1,\mathbf x_2$
represent the same Spin$^c$--structure in $\mathrm{Spin}^c(Y)$.
Hence there is a topological Whitney disk $\phi$ for $\mathbb
T_{\alpha},\mathbb T_{\beta}$ connecting $\mathbf x_1$ to $\mathbf
x_2$. Since the $\beta_1$--components of $\mathbf x_1$ and
$\mathbf x_2$ are both $x$, $\partial\phi$ contains
$n_w(\phi)-n_z(\phi)$ copies of $\beta_1$. Let
$\psi^d=\psi_1-\psi_2$. (See Figure~1 for an illustration.) By
(\ref{eq:SpinEq}), we have
$$\underline{\mathfrak s}_{w,z}(\mathbf
x_1)-\underline{\mathfrak s}_{w,z}(\mathbf
x_2)=-\big(n_w(\psi^d)-n_z(\psi^d)\big)\cdot\mu,$$ thus
$$n_w(\phi)-n_z(\phi)=-\big(n_w(\psi^d)-n_z(\psi^d)\big).$$
So we can glue $\phi$ and $\psi^d$  together to get a disk
$\varphi'$. After a Hamiltonian translation, $\varphi'$ becomes a
topological Whitney disk $\varphi$ for $\mathbb T_{\alpha},\mathbb
T_{\gamma}$, connecting $\mathbf x'_1$ to $\mathbf x'_2$. Hence
$\mathfrak t_1=\mathfrak t_2$.
\end{proof}

The following result is the twisted version of
\cite[Theorem~4.1]{OSzRatSurg}. We do not state it in the most
generality since the current version suffices for our purpose.

\begin{prop}\label{prop:LargeSurg}
Let $K\subset Y$ be a rationally null-homologous knot in a closed,
oriented three-manifold, equipped with a framing $\lambda$, and
let $\omega$ be a $1$--cycle in $Y-K$. Let $\widehat
A_{\xi}(Y,K,\omega)=\underline{C}_{\xi}\big\{\max\{i,j\}=0\big\}$.
Then, for all sufficiently large $m$ and all $\mathfrak
t\in\mathrm{Spin}^c(Y_{m\mu+\lambda}(K))$, there is an isomorphism
$$\Psi_{\mathfrak t,m}\co \underline{\widehat{CF}}(Y_{m\mu+\lambda}(K),\mathfrak t,\omega;\Lambda)
\to\widehat A_{\Xi(\mathfrak t)}(Y,K,\omega).$$
\end{prop}
\begin{proof}
See \cite[Theorem~4.1]{OSzRatSurg}.
\end{proof}

Another result we will need is the following twisted version of
\cite[Corollary~5.3]{OSzRatSurg}.

\begin{prop}\label{prop:Kunneth}
If $K_2\subset Y_2$ is a $U$--knot, $\omega$ is a $1$--cycle in
$Y_1-K_1$, then for each
$\xi_1\in\underline{\mathrm{Spin}^c}(Y_1,K_1)$ and $\mathfrak
s_2\in \mathrm{Spin}^c(Y_2)$, there is some
$\xi_2\in\underline{\mathrm{Spin}^c}(Y_2,K_2)$ representing
$\mathfrak s_2$, with the property that
$$\underline{CFK^{\infty}}(Y_1,K_1,\omega,\xi_1)\cong\underline{CFK^{\infty}}(Y_1\#Y_2,K_1\#K_2,\omega,\xi_1\#\xi_2)$$
as $\mathbb Z\oplus\mathbb Z$--filtered chain complexes.
\end{prop}

\section{Properties of twisted Heegaard Floer homology}

In this section, we collect some properties of twisted Heegaard
Floer homology. In particular, we prove some nontriviality results
following \cite{OSzGenus}.

\subsection{Surgery exact sequences}

As in \cite{OSzAnn2}, there are surgery exact sequences for
twisted Heegaard Floer homology. One of them is as follows (see
also \cite{AiP}).

\begin{prop}\label{prop:Exact}
Suppose $K\subset Y$ is a knot with frame $\lambda$, and
$\omega\subset Y-K$ is a $1$--cycle, then $\omega$ also lies in
the manifolds $Y_{\lambda}$ and $Y_{\lambda+\mu}$ obtained by
surgeries on $K$. The $2$--handle addition cobordism $W$ from $Y$
to $Y_{\lambda}$ naturally contains $\omega\times I$. We can
define a chain map induced by $W$:
$$\underline{f^{\infty}_{W,\:\omega\times I}}\co\underline{CF^{\infty}}(Y,\omega;\Lambda)
\to\underline{CF^{\infty}}(Y_{\lambda},\omega;\Lambda).$$
Similarly, there are two other chain maps induced by the
cobordisms $Y_{\lambda}\to Y_{\lambda+\mu}$ and
$Y_{\lambda+\mu}\to Y$. We then have the long exact sequence :
$$
\begin{CD}
\cdots\to\underline{HF^+}(Y,\omega;\Lambda)\to
\underline{HF^+}(Y_{\lambda},\omega;\Lambda)\to
\underline{HF^+}(Y_{\lambda+\mu},\omega;\Lambda)\to\cdots.
\end{CD}
$$
\end{prop}

\begin{lem}\label{lem:CutReglue}
Suppose $F$ is a closed surface in a closed manifold $Y$, and
$F_0$ is a component of $F$ such that its genus $\ge2$. Let $Y'$
be the manifold obtained by cutting open $Y$ along $F_0$ and
regluing by a self-homeomorphism $\varphi$ of $F_0$. It is
well-known that $\varphi$ can be realized by a product of Dehn
twists along a set of curves $\mathcal C$ on $F_0$. Let $\omega$
be a $1$--cycle in $Y$ such that $\omega$ is disjoint from
$\mathcal C$, then $\omega$ can also be viewed as a $1$--cycle in
$Y'$. Then we have
$$\underline{HF^+}(Y,\omega,[F],\frac{x(F)}2;\Lambda)\cong \underline{HF^+}(Y',\omega,[F],\frac{x(F)}2;\Lambda).$$

Similarly, suppose $F$ is a Seifert-like surface for a knot $K$ in
a closed manifold $Y$, and $F_0$ is a component of $F$ such that
its genus $\ge2$. Let $Y'$ be the manifold obtained by cutting
open $Y$ along $F_0$ and regluing by a self-homeomorphism
$\varphi$, which can be realized by a product of Dehn twists along
a set of curves $\mathcal C$ on $F_0$. Let $\omega$ be a
$1$--cycle in $Y$ such that $\omega$ is disjoint from $\mathcal
C$, then $\omega$ can also be viewed as a $1$--cycle in $Y'$. The
new knot in $Y'$ is still denoted by $K$. Then we have
$$\underline{\widehat{HFK}}(Y,K,\omega,[F],\frac{x(F)+1}2;\Lambda)\cong \underline{\widehat{HFK}}(Y',K,\omega,[F],\frac{x(F)+1}2;\Lambda).$$
\end{lem}
\begin{proof}
The proof is a standard application of the surgery exact sequence
and the adjunction inequality.
\end{proof}

\subsection{The presence of a non-separating sphere}

When there is a non-separating two-sphere, we have the following
properties from \cite{NiClosed}.

\begin{lem}{\bf\cite[Lemma 2.1]{NiClosed}}\label{lem:TrivialS2}
Suppose $Y$ contains a non-separating two-sphere $S$,
$\omega\subset Y$ is a closed curve such that $\omega\cdot S\ne0$.
We then have
$$\underline{\widehat{HF}}(Y,\omega;\Lambda)=0,\quad\underline{HF^+}(Y,\omega;\Lambda)=0.$$
\end{lem}

\begin{lem}{\bf\cite[Lemma 5.1]{NiClosed}}\label{lem:0surgery}
Suppose $Y$ is a closed $3$--manifold containing a non-separating
two-sphere $S$, $K\subset Y$ is a null-homologous knot, $F$ is a
 Seifert-like surface for $K$. Let $Y_0(K)$ be the
manifold obtained by doing $0$--surgery on $K$, and let
$\widehat{F}$ be the extension of $F$ in $Y_0(K)$. Let
$\omega\subset Y-K$ be a $1$--cycle such that $\omega\cdot S\ne0$.
We then have
$$\underline{\widehat{HFK}}(Y,K,\omega,[F],\frac{x(F)+1}2;\Lambda)\cong \underline{HF^+}(Y_0(K),\omega,[\widehat F],\frac{x(F)-1}2;\Lambda).$$
\end{lem}

\subsection{The topmost nontrivial term}

In this subsection we will prove some nontriviality results
following the approach in \cite{OSzGenus}. Although it is possible
to prove stronger results, we are satisfied with the current
version since it is sufficient for our purpose. We also cite a
result about twisted Floer homology and fibered knots.

\begin{lem}\label{lem:SymplFil}
Suppose $Y$ is a closed $3$--manifold with a taut foliation
$\mathscr F$ which is smooth except possibly along some compact
leaves. Then $\mathscr F$ can be approximated by a positive
contact structure $\xi_+$ and a negative contact structure $\xi_-$
, and there is a nonempty open subset $U^*\subset H^2(Y;\mathbb
R)$ with the following property: for any $h\in U^*$, there exists
a symplectic form $\Omega$ on $Y\times[-1,1]$, such that
$[\Omega]=h\in H^2(Y;\mathbb R)$, $\Omega|_{Y\times\{\pm1\}}$ is
everywhere positive on $\xi_{\pm}$.
\end{lem}
\begin{proof}
By \cite{ET}, we can approximate $\mathscr F$ by a positive
contact structure $\xi_+$ and a negative contact structure
$\xi_-$, and there exists a symplectic form $\Omega$ on
$Y\times[-1,1]$, $\Omega|_{Y\times\{\pm1\}}$ is everywhere
positive on $\xi_{\pm}$. Now if we perturb $\Omega$ by a small
closed $2$--form on $Y\times[-1,1]$, we still get a symplectic
form which is everywhere positive on $\xi_{\pm}$. This finishes
the proof.
\end{proof}

\begin{thm}\label{thm:TwistNorm}
Suppose $Y$ is a closed irreducible $3$--manifold, $F$ is a taut
surface in $Y$. Then there exists a nonempty open set $U\subset
H_1(Y;\mathbb R)$, such that for any $\omega\in U$,
$$\underline{HF^+}(Y,\omega,[F],\frac12x(F);\Lambda)\ne0, \quad \underline{\widehat{HF}}(Y,\omega,[F],\frac12x(F);\Lambda)\ne0.$$
\end{thm}
\begin{proof}
By \cite{G1}, there exists a taut foliation $\mathscr F$ of $Y$,
such that $F$ is a union of compact leaves of $\mathscr F$, and
$\mathscr F$ is smooth except possibly along toral components of
$F$. By Lemma~\ref{lem:SymplFil} we have a nonempty open subset
$U^*\subset H^2(Y;\mathbb R)$ with the property stated there. Let
$U\subset H_1(Y;\mathbb R)$ be the dual of $U^*$. Now for any
$\omega\in U$, the argument in \cite[Section 4]{OSzGenus} shows
that $\underline{HF^+}(Y,\omega,[F],\frac12x(F);\Lambda)\ne0$.

If $\underline{\widehat{HF}}(Y,\omega,[F],\frac12x(F);\Lambda)=0$,
then the map
$$U\co \underline{HF^+}(Y,\omega,[F],\frac12x(F);\Lambda)\to \underline{HF^+}(Y,\omega,[F],\frac12x(F);\Lambda)$$
is an isomorphism. Since
$\underline{HF^+}(Y,\omega,[F],\frac12x(F);\Lambda)\ne0$ and
$U^na=0$ for any
$a\in\underline{HF^+}(Y,\omega,[F],\frac12x(F);\Lambda)$ and
sufficiently large $n$, we get a contradiction.
\end{proof}

The following lemma will be used in the proof of
Theorem~\ref{thm:TwistGenus}.

\begin{lem}\label{lem:LongiFol}
Suppose $K\subset Y$ is a null-homologous knot, $Y-K$ is
irreducible, $F$ is a taut Seifert-like surface for $K$. Let
$J\subset S^3$ be a fibered knot with fiber $G$, and let $F'$ be
the Seifert-like surface for $K\#J$ which is the boundary
connected sum of $F$ and $G$. Then if the genus of $J$ is
sufficiently large, $Y-\ond(K\#J)$ admits a smooth longitudinal
foliation such that $F'$ is a union of compact leaves.
\end{lem}
\begin{proof}
The proof is the same as \cite[Proposition~2.4]{NiNorm}.
\end{proof}

\begin{thm}\label{thm:TwistGenus}
Suppose $K$ is a null-homologous knot in a closed 3--manifold $Y$,
$Y-K$ is irreducible. Let $F$ be a taut Seifert-like surface for
$K$. Then there exists a nonempty open set $U\subset H_1(Y;\mathbb
R)$, such that for any $\omega\in U$,
$$\underline{\widehat{HFK}}(Y,K,\omega,[F],\frac{x(F)+1}2;\Lambda)\ne0.$$
\end{thm}
\begin{proof}
By Lemma~\ref{lem:LongiFol}, the complement of $K_1=K\#J$ admits a
smooth longitudinal foliation with a compact leaf $F'$. So
$Y_0(K_1)$ admits a taut smooth foliation with a compact leaf
$\widehat{F_1}$. By Theorem~\ref{thm:TwistNorm}, there exists a
nonempty open set $U_1\subset H_1(Y_0(K_1);\mathbb R)$, such that
for any $\omega\in U_1$,
$\underline{HF^+}(Y_0(K_1),\omega,[\widehat{F_1}],\frac{x(F_1)-1}2)\ne0$.

Let $(M,\gamma)$ be the sutured manifold obtained by cutting $Y$
open along $F_1$. Since $J$ is a fibered knot, $(M,\gamma)$
contains a non-separating product disk $D$. We can cut $Y$ open
along $F_1$ then reglue by a diffeomorphism $\varphi$ to get a new
knot $K'$ in a new manifold $Y'$, such that $D\cap R_+(\gamma)$
and $D\cap R_-(\gamma)$ are glued together. Now $D$ becomes a
non-separating annulus $A$ in the complement of $K'$, such that
$\partial A$ consists of two copies of the meridian of $K'$. So
$Y'$ contains a non-separating sphere $S$.

The diffeomorphism $\varphi$ can be realized by a product of Dehn
twists along a set  of curves $\mathcal C$ on $F_1$. In other
words, there exists a link $\mathcal L\subset Y-K_1$, such that a
Dehn surgery on $\mathcal L$ yields $Y'$. Let
\begin{eqnarray*}
&&\rho\co H_1(Y-K_1-\mathcal L;\mathbb R)\to H_1(Y-K_1;\mathbb
R)=H_1(Y_0(K_1);\mathbb R),\\
&&\rho'\co H_1(Y-K_1-\mathcal L;\mathbb R)\to H_1(Y'-K';\mathbb
R)=H_1(Y'_0(K');\mathbb R)
\end{eqnarray*}
 be the natural inclusion maps. Both
$\rho$ and $\rho'$ are surjective. Let $$\mathcal V\subset
H_1(Y'-K';\mathbb R)$$ be the codimension $1$ subspace defined by
$$v\cdot [S]=0.$$

Let $\omega$ be a $1$--cycle in $Y-K_1-\mathcal L$ such that
$\omega\in\rho^{-1}(U)-(\rho')^{-1}(\mathcal V)$. By
Lemma~\ref{lem:CutReglue}, we have
$$\underline{\widehat{HFK}}(Y,K_1,\rho(\omega),[F_1],\frac{x(F_1)+1}2)\cong
\underline{\widehat{HFK}}(Y',K',\rho'(\omega),[F_1],\frac{x(F_1)+1}2),$$
$$\underline{HF^+}(Y_0(K_1),\rho(\omega),[\widehat{F_1}],\frac{x(F_1)-1}2)\cong
\underline{HF^+}(Y'_0(K'),\rho'(\omega),[\widehat{F_1}],\frac{x(F_1)-1}2).
$$
By Lemma~\ref{lem:0surgery},
$$\underline{\widehat{HFK}}(Y',K',\rho'(\omega),[F_1],\frac{x(F_1)+1}2)\cong
\underline{HF^+}(Y'_0(K'),\rho'(\omega),[\widehat{F_1}],\frac{x(F_1)-1}2),$$
Hence
$\underline{\widehat{HFK}}(Y,K_1,\rho(\omega),[F_1],\displaystyle\frac{x(F_1)+1}2;\Lambda)\ne0$.

By the connected sum formula
$$\underline{\widehat{HFK}}(Y,K\#J,\rho(\omega);\Lambda)\cong \underline{\widehat{HFK}}(Y,K,\rho(\omega);\Lambda)\otimes{\widehat{HFK}(S^3,J;\mathbb R)},$$
we conclude that for any $\omega\in
\rho^{-1}(U)-(\rho')^{-1}(\mathcal V)$,
$\underline{\widehat{HFK}}(Y,K,\rho(\omega),g)\ne0$. Let
$$i_*\co H_1(Y_0(K);\mathbb R)=H_1(Y-K;\mathbb R)\to H_1(Y;\mathbb R)$$
be the natural map which is a projection. Let
$$U=i_*\rho\big(\rho^{-1}(U)-(\rho')^{-1}(\mathcal V)\big),$$ then
$U$ is the nonempty open set we need.
\end{proof}

The following result is a twisted version of a theorem due to
 Ghiggini \cite{Gh} and the author \cite{NiFibred}.

\begin{thm}{\bf\cite[Theorem 2.2]{NiClosed}}\label{thm:TwistFibre}
Suppose $K$ is a null-homologous knot in a closed, oriented,
connected 3--manifold $Y$, $Y-K$ is irreducible, and $F$ is a
genus $g$ Seifert surface for $K$. Let $\omega\subset Y-K$ be a
$1$--cycle. If
$$\underline{\widehat{HFK}}(Y,K,\omega,[F],g;\Lambda)\cong\Lambda,$$
then $K$ is fibered, and $F$ is a fiber of the fibration.
\end{thm}

\section{Property G}

This section is devoted to the proof of Theorem~\ref{thm:PropG},
which is a direct corollary of the properties listed in the last
section.

\begin{proof}[Proof of Theorem~\ref{thm:PropG}]
We first prove Property G1. If $F$ is a taut Seifert-like surface
for $K$, by Theorem~\ref{thm:TwistGenus} we can find a $1$--cycle
$\omega\subset Y-K$, such that $\omega\cdot S\ne0$ and
$$\underline{\widehat{HFK}}(Y,K,\omega,[F],\frac{x(F)+1}2;\Lambda)\ne0.$$
Now Lemma~\ref{lem:0surgery} implies that
$$\underline{HF^+}(Y_0(K),\omega,[\widehat F],\frac{x(F)-1}2;\Lambda)\ne0,$$
hence $\widehat F$ is taut.

Now we prove Property G2. Suppose $Y_0(K)$ fibers over $S^1$ with
fiber in the homology class $[\widehat F]$, where $F$ is a taut
Seifert-like surface for $K$. By Property G1, $\widehat F$ is taut
in $Y_0(K)$, hence $\widehat F$ is isotopic to a fiber of the
fibration. Choose a $1$--cycle $\omega\subset Y-K$, such that
$\omega\cdot S\ne0,\omega\cdot[\widehat F]\ne0$. Since $Y_0(K)$
fibers over $S^1$, by \cite{AiP} we have
$$\underline{HF^+}(Y_0(K),\omega,[\widehat{F}],g(F)-1;\Lambda)\cong\Lambda.$$
Lemma~\ref{lem:0surgery} then implies that
$$\underline{\widehat{HFK}}(Y,K,\omega,[F],g(F);\Lambda)\cong\Lambda.$$
Using Theorem~\ref{thm:TwistFibre}, we conclude that $K$ is
fibered with fiber $F$.
\end{proof}

\section{Cosmetic surgery}

In this section, we will prove Theorem~\ref{thm:Cosmetic}. Like
\cite[Section~9]{OSzRatSurg}, the proof relies on the rational
surgery formula of Floer homology. However, our situation here is
much simpler. The result we will use is as follows.

\begin{prop}\label{prop:pqSurg}
Let $Y$ be a closed $3$--manifold that contains a non-separating
sphere $S$. $K$ is a null-homologous knot in $Y$, such that $Y-K$
is irreducible. Let $\omega$ be a $1$--cycle in $Y-K$ satisfying
$\omega\cdot S\ne0$. Then there exists a constant
$R=R(Y,K,\omega)$, such that
$$\mathrm{rank}_{\Lambda}\underline{\widehat{HF}}(Y_{\frac pq}(K),\omega;\Lambda)=qR$$
for any $\frac pq\in \mathbb Q$. Here $p,q\in\mathbb Z, q>0,
\gcd(p,q)=1$.
\end{prop}

The following lemma is an analogue of
\cite[Theorem~6.1]{OSzRatSurg}.

\begin{lem}\label{lem:MorseSurg}
Let $Y$ be a closed $3$--manifold that contains a non-separating
sphere $S$. $K$ is a rationally null-homologous knot in $Y$,
$\lambda$ is a frame on $K$. Let $\omega$ be a $1$--cycle in $Y-K$
satisfying $\omega\cdot S\ne0$. Let
$$\widehat{\mathbb A}(Y,K,\omega)=\bigoplus_{\xi\in\underline{\mathrm{Spin}^c}(Y,K)}\widehat
A_{\xi}(Y,K,\omega).$$ Then there is an isomorphism
$$\underline{\widehat{HF}}(Y_{\lambda}(K),\omega;\Lambda)\cong H_*(\widehat{\mathbb A}(Y,K,\omega)).$$
\end{lem}
\begin{proof}
We claim that for any two frames $\lambda_1,\lambda_2$ on $K$,
$$\underline{\widehat{HF}}(Y_{\lambda_1}(K),\omega;\Lambda)\cong\underline{\widehat{HF}}(Y_{\lambda_2}(K),\omega;\Lambda).$$
This claim follows from Proposition~\ref{prop:Exact} and the fact
that
$$\underline{\widehat{HF}}(Y,\omega;\Lambda)=0.$$

By the above claim and Proposition~\ref{prop:LargeSurg}, when $m$
is sufficiently large we have
\begin{equation}\label{eq:oplusA}
\underline{\widehat{HF}}(Y_{\lambda}(K),\omega;\Lambda)\cong\bigoplus_{\mathfrak
t\in\mathrm{Spin}^c(Y_{m\mu+\lambda})} \widehat A_{\Xi(\mathfrak
t)}(Y,K,\omega).
\end{equation}

Recall that $\widehat
A_{\xi}(Y,K,\omega)=\underline{C}_{\xi}\big\{\max\{i,j\}=0\big\}$.
By (\ref{eq:Complex}), $\widehat A_{\xi}(Y,K,\omega)\ne0$ only if
some Spin$^c$--structure $\xi+n\mathrm{PD}[\mu]$ is represented by
an intersection point $\mathbf x\in\mathbb T_{\alpha}\cap\mathbb
T_{\beta}$. Moreover, by (\ref{eq:Complex}) and
Lemma~\ref{lem:Forget}, there exists a constant $N_0$, such that
for any $\mathbf x$, if $|n|>N_0$, then $$\widehat
A_{\underline{\mathfrak s}_{w,z}(\mathbf
x)-n\mathrm{PD}[\mu]}(Y,K,\omega)\cong
\underline{\widehat{CF}}(Y,\omega,\mathfrak r)$$ for some Spin$^c$
structure $\mathfrak r$ depending on $\mathbf x,n$. By
Lemma~\ref{lem:TrivialS2}, the right hand side of the above
equation is $0$.

The analysis in the last paragraph shows that, if $m$ is
sufficiently large, then the image of $\Xi$ contains all the $\xi$
such that $\widehat A_{\xi}(Y,K,\omega)\ne0$. Our desired result
then follows from (\ref{eq:oplusA}) and Lemma~\ref{lem:XiInj}.
\end{proof}

Let $K$ be a null-homologous knot in $Y$. As in \cite[Section
7]{OSzRatSurg}, $Y_{\frac pq}(K)$ can be realized by a Morse
surgery on the knot $K'=K\#O_{q/r}\subset Y'=Y\#L(q,r)$, where
$O_{\frac qr}$ is a $U$--knot in $L(q,r)$.

Suppose $\xi\in\underline{\mathrm{Spin}^c}(Y',K')$, then $\xi$ is
the connected sum of two Spin$^c$--structures
$\xi_1\in\underline{\mathrm{Spin}^c}(Y,K)$ and
$\xi_2\in\underline{\mathrm{Spin}^c}(L(q,r),O_{q/r})$. Let
$\underline{\Pi_1}(\xi)=\xi_1$,
$\Pi_2(\xi)\in\mathrm{Spin}^c(L(q,r))$ be the Spin$^c$--structure
represented by $\xi_2$.

\begin{lem}\label{lem:SpinConn}
There is a bijective map
$$\underline{\Pi_1}\times\Pi_2\co \underline{\mathrm{Spin}^c}(Y',K')\to
\underline{\mathrm{Spin}^c}(Y,K)\times\mathrm{Spin}^c(L(q,r)).$$
\end{lem}
\begin{proof}
It suffices to show that there is a natural short exact sequence
\begin{equation}\label{eq:ShortEx}
0\to H^2(Y,K)\to H^2(Y',K')\to H^2(L(q,r))\to0.
\end{equation}

There is a $2$--sphere which splits $Y'$ into two parts $U_1,U_2$,
where $U_1=Y-B^3$, $U_2=L(q,r)-B^3$. The knot $K'$ is split into
two arcs $L_1\subset U_1,L_2\subset U_2$. Now there is a
Mayer-Vietoris sequence
\[\begin{array}{cccccc}

&&\cdots&\to&H^1(U_2\cup L_1,K')&\to\\
H^2(Y',U_2\cup L_1)&\to& H^2(Y',K')&\to& H^2(U_2\cup L_1,K')&\to\\
H^3(Y',U_2\cup L_1)&\to&\cdots&&&.
\end{array}
\] By the Excision Axiom, we have $$H^*(U_2\cup L_1,K')\cong
H^*(U_2,L_2),\quad H^*(Y',U_2\cup L_1)\cong H^*(Y,K).$$ Moreover,
it is easy to compute $$H^2(U_2,L_2)\cong\mathbb Z_q\cong
H^2(L(q,r)),\quad H^1(U_2,L_2)=0=H^3(Y,K).$$ So we have the
natural exact sequence (\ref{eq:ShortEx}).
\end{proof}

\begin{proof}[Proof of Proposition~\ref{prop:pqSurg}] Using
 Proposition~\ref{prop:Kunneth} and
Lemma~\ref{lem:SpinConn}, we conclude that $\widehat{\mathbb
A}(Y',K',\omega)$ is the direct sum of $q$ copies of
$\widehat{\mathbb A}(Y,K,\omega)$. Hence
$$\mathrm{rank}_{\Lambda}\underline{\widehat{HF}}(Y_{\frac pq}(K),\omega;\Lambda)=q\cdot\mathrm{rank}_{\Lambda}H_*(\widehat
{\mathbb A}(Y,K,\omega))$$ by Lemma~\ref{lem:MorseSurg}.
\end{proof}

Theorem~\ref{thm:Cosmetic} does not directly follow from the
previous two results. The reason is that
$\underline{\widehat{HF}}(Y,\omega;\Lambda)$ is not an invariant
for $Y$: it depends on the choice of $\omega$. However, it is not
hard to overcome this difficulty.

\begin{proof}[Proof of Theorem~\ref{thm:Cosmetic}]
Assume there are two rational numbers
$\frac{p_1}{q_1},\frac{p_2}{q_2}$ satisfying that there is a
homeomorphism
$$f\co Y_{\frac{p_1}{q_1}}\to\pm Y_{\frac{p_2}{q_2}},$$ then
$|p_1|=|p_2|$ for homological reason. If
$\frac{p_1}{q_1}\ne\frac{p_2}{q_2}$, then we can assume
$$0<q_1<q_2.$$

By Lemma~\ref{lem:0surgery} and Theorem~\ref{thm:TwistGenus},
there exists a non-empty set $U\subset H_1(Y;\mathbb R)$, such
that for any $1$--cycle $\omega\subset Y-K$ representing an
element in $U$, one has
$\underline{HF^+}(Y_0,\omega;\Lambda)\ne0$. Thus
$\underline{\widehat{HF}}(Y_0,\omega;\Lambda)\ne0$ as argued in
the proof of Theorem~\ref{thm:TwistNorm}.

Since $K$ is null-homologous, we can identify $H_1(Y;\mathbb R)$
with $H_1(Y_r;\mathbb R)$ for any $r\in\mathbb Q-\{0\}$. Let
$\mathcal V$ be the subspace of $H_1(Y;\mathbb R)$ defined by the
equation $x\cdot S=0$. Choose an $$\omega\in
U\backslash\cup_{n\in\mathbb Z}f^n_*(\mathcal V),$$ then
$(f^n_*\omega)\cdot S\ne0$ for any $n\in\mathbb Z$. By
Proposition~\ref{prop:pqSurg},
$$
\mathrm{rank}_{\Lambda}\underline{\widehat{HF}}(Y_{\frac{p_1}{q_1}},f^n_*\omega;\Lambda)=
\frac{q_1}{q_2}\mathrm{rank}_{\Lambda}\underline{\widehat{HF}}(Y_{\frac{p_2}{q_2}},f^n_*\omega;\Lambda)
$$
for any $n\in\mathbb Z$. Moreover, since $f\co
Y_{\frac{p_1}{q_1}}\to\pm Y_{\frac{p_2}{q_2}}$ is a homeomorphism,
using Proposition~\ref{prop:OrRev} if necessary, we have
$$
\mathrm{rank}_{\Lambda}\underline{\widehat{HF}}(Y_{\frac{p_1}{q_1}},f^n_*\omega;\Lambda)=
\mathrm{rank}_{\Lambda}\underline{\widehat{HF}}(Y_{\frac{p_2}{q_2}},f^{n+1}_*\omega;\Lambda).
$$
Thus we get
$$\mathrm{rank}_{\Lambda}\underline{\widehat{HF}}(Y_{\frac{p_1}{q_1}},f^n_*\omega;\Lambda)=\left(\frac{q_1}{q_2}\right)^n
\mathrm{rank}_{\Lambda}\underline{\widehat{HF}}(Y_{\frac{p_1}{q_1}},\omega;\Lambda).$$
By Proposition~\ref{prop:pqSurg},
$$\mathrm{rank}_{\Lambda}\underline{\widehat{HF}}(Y_{\frac{p_1}{q_1}},\omega;\Lambda)=
q_1\mathrm{rank}_{\Lambda}\underline{\widehat{HF}}(Y_0,\omega;\Lambda)\ne0,$$
so
$0<\mathrm{rank}_{\Lambda}\underline{\widehat{HF}}(Y_{\frac{p_1}{q_1}},f^n_*\omega;\Lambda)<1$
 when $n$ is sufficiently large, a contradiction.
\end{proof}

\end{document}